\documentclass[11pt,english]{article}
\usepackage{amsmath}
\usepackage{amssymb}
\usepackage{amsthm}
\usepackage[english]{babel}
\usepackage{times}
\usepackage[T1]{fontenc}
\usepackage[all]{xy}
\usepackage{geometry}
\usepackage{color}
\usepackage{authblk}
\theoremstyle{plain}
\newtheorem{teo}{Theorem}[section]
\newtheorem{coro}[teo]{Corollary}
\newtheorem{lem}[teo]{Lemma}
\newtheorem{pro}[teo]{Proposition}

\newtheorem*{teonn}{Theorem}

\theoremstyle{definition}
\newtheorem{oss}[teo]{Remark}

\newcommand{\D}{\mathbb{D}}
\newcommand{\B}{\mathbb{B}}
\newcommand{\HH}{\mathbb{H}}
\newcommand{\N}{\mathbb{N}}
\newcommand{\C}{\mathbb{C}}
\newcommand{\rr}{\mathbb{R}}
\newcommand{\s}{\mathbb{S}}
\newcommand{\Z}{\mathbb{Z}}

\DeclareMathOperator{\Span}{span}

\newcommand{\clK}{\mathcal{K}}

\newcommand{\p}{\partial}

\DeclareMathOperator{\RRe}{Re}

\DeclareMathOperator{\ess}{ess}

\newcommand{\RR}{\mathbb{R}}
\newcommand{\BB}{\mathbb{B}}

\renewcommand{\SS}{\mathbb{S}}

 \newcommand{\Addresses}{{
  \bigskip
  \footnotesize

  \textsc{A.~Monguzzi,} \textsc{Dipartimento di Matematica, Universit\`a degli Studi di Milano, Via C. Saldini 50, 20133 Milano, Italy.}\par\nopagebreak
  \textit{E-mail address:} \texttt{alessandro.monguzzi@unimi.it}

  \medskip

  \textsc{G.~Sarfatti,} \textsc{Dipartimento di Matematica e Informatica ``U. Dini'', Universit\`a di Firenze, Viale \mbox{Morgagni} 67/A, 501334 Firenze, Italy.}\par\nopagebreak
  \textit{E-mail address:}  \texttt{giulia.sarfatti@unifi.it}

}}
 
\title{\bf Shift invariant subspaces of slice $L^2$ functions}

\author[1]{Alessandro Monguzzi\thanks{Partially supported by the 2015 PRIN grant
  \emph{Real and Complex Manifolds: Geometry, Topology and Harmonic Analysis}  
  of the Italian Ministry of Education (MIUR)}, Giulia Sarfatti \thanks{Partially supported by INDAM-GNSAGA, by the 2014 SIR grant {\em Analytic Aspects in Complex and Hypercomplex Geometry} and by Finanziamento Premiale FOE 2014 {\em Splines for accUrate NumeRics: adaptIve models for Simulation Environments} of the Italian Ministry of Education (MIUR)}}

\date{}

\begin{document}

\maketitle

\begin{abstract}
In this paper we characterize the closed invariant subspaces for the ($*$-)multiplier operator of the quaternionic space of slice $L^2$ functions. As a consequence, we obtain the inner-outer factorization theorem for the quaternionic Hardy space on the unit ball and we provide a characterization of quaternionic outer functions in terms of cyclicity.
\end{abstract}

\section{Introduction and main results}

\medskip

Given a Hilbert space $\mathcal H$ and a bounded operator $T:\mathcal H\to\mathcal H$ it is a classical and challenging problem to investigate and characterize the closed invariant subspaces for the operator $T$. We say that a closed subspace $\mathcal K\subseteq \mathcal H$ is invariant for $T$ if $T\mathcal K\subseteq \mathcal K$.

Let $\D=\{z\in\C:|z|<1\}$ denote the unit disc in the complex plane $\C$ and consider the Hilbert space $\mathcal H= L^2(\p \D)$, that is, the space of square-integrable functions on the unit circle. Then, the invariant subspaces for the multiplier operator  
$$
M_z(f)(z):=zf(z)
$$
are completely characterized as we shall see shortly.

Let us consider the Hardy space of the unit disc, that is, the function space
$$
H^2(\D)=\left\{f\textrm{ holomorphic in $\D$}: f(z)=\sum_{n\in\N} a_nz^n\textrm{ and }\{a_n\}_{n\in\N} \in \ell^2(\N,\C) \right\},
$$
It is a well-known fact that each function $f\in H^2(\D)$ admits a boundary value function, which we still denote by $f$, in $L^2(\p\D)$. The set of all these boundary value functions turns out to be a closed subspace of $L^2(\p \D)$ which we denote by $H^2(\p \D)$. This latter space can in turn be identified with $\ell^2(\N, \C)$, the space of square-summable sequences over the non-negative integers. It is clear that the multiplier operator $M_z$ acting on $H^2(\p\D)$ is a model for the right-shift operator $(a_0,a_1,\ldots)\mapsto(0,a_0,a_1,\ldots)$ on $\ell^2(\N, \C)$. Therefore, the invariant subspaces for $M_z$ of $H^2(\p\D)$ are frequently called shift invariant subspaces. More generally, from now on we will call (shift) invariant subspaces the closed subspaces of $L^2(\p \D)$ that are invariant for the operator $M_z$. Moreover, we say that $\clK\subseteq L^2(\p\D)$ is \emph{doubly invariant} if $M_z\clK=\clK$, whereas we say that $\clK$ is \emph{simply invariant} if $M_z\clK\subsetneq \clK$. The following theorem holds.

\begin{teonn}[{\cite[\dots]{wiener, beu, helson}}]
 Let $\mathcal K$ be a closed subspace of $L^2(\p \D)$. Then,
 \begin{itemize}
  \item[(i)] $\mathcal K$ is a doubly invariant subspace if and only if there exists a (unique) measurable set $E\subseteq\p\D$ such that $\mathcal K=\chi_E L^2(\p \mathbb D)$, where $\chi_E$ denotes the characteristic function of the set $E$;  
  \item[(ii)] $\mathcal K$ is a simply invariant subspace if and only if there exists a measurable function $\varphi$, unique up to a multiplicative constant of modulus 1, such that $|\varphi|=1$ almost everywhere on $\p\D$ and $\mathcal K=\varphi H^2(\p\D)$.
 \end{itemize}
\end{teonn}
\noindent This result is due to several mathematicians and, in addition to the works we already mentioned, we refer the reader also to \cite[Chapter I]{nikolski} for another proof of the theorem and a detailed account of its history.\\
As a corollary of the previous result, the celebrated theorem of Beurling follows.
\begin{teonn}[\cite{beu}]
Let $\mathcal K$ be a non-zero closed subspace of $H^2(\p\D)$ that is invariant for the multiplier operator $M_z$. Then there exists $\varphi\in H^2(\p \D)$ such that $|\varphi|=1$ almost everywhere on $\p \D$ and $\mathcal K=\varphi H^2(\p \D)$.
\end{teonn}
Since Beurling's paper the problem of characterizing the invariant subspaces of spaces of holomorphic functions different from the Hardy space has been challenged by several mathematicians. Among others, we recall the works \cite{richter,luo-richter}, where the case of the Dirichlet space of the unit disc is studied, the work \cite{aleman-richter-sundberg} for the Bergman space case and \cite{richter-sunkes, arcozzi-levi} for the case of the Drury--Arveson space of the unit ball.  

In this paper we work in the quaternionic setting and we study the invariant subspaces of the space of slice $L^2$ functions. \\
Let $\HH$ denotes the skew field of quaternions, let $\BB=\{q\in\HH:\ |q|<1\}$ be the quaternionic unit ball and
let $\partial \BB$ be its boundary, containing elements of the form $q=e^{tI},\ I\in\SS,\ t\in\RR$, where $\SS=\{q\in \HH : q^2=-1\}$ is the two dimensional sphere of imaginary units in $\HH$. Then,
\[ \HH=\bigcup_{I\in \s}(\rr+\rr I),  \hskip 1 cm \rr=\bigcap_{I\in \s}(\rr+\rr I),\]
where the ``slice'' $L_I:=\rr+\rr I$ can be identified with the complex plane $\C$ for any $I\in\s$.  \\
We endow $\partial\BB$ with the measure
\begin{equation}\label{Sigma}
d\Sigma\left(e^{tI}\right)=d\sigma(I)dt,
\end{equation}
where $dt$ is the Lebesgue measure on $\RR$ and $d\sigma$ is the standard area element of $\s$, normalized so that $\sigma(\s)=\Sigma(\p\B)=1$. The measure $d\Sigma$ is naturally associated with the Hardy space $H^2(\BB)$ of slice regular functions on $\BB$, see \cite{deFGS, metrica}. 

We recall that a function $f:\BB\to\HH$ is a \emph{slice regular function} if the restriction $f_I$ of $f$ to $\BB\cap L_I$
is {holomorphic}, i.e., it has continuous partial derivatives and it is such that
\[\overline{\p}_If_I(x+yI)=\frac{1}{2}\left(\frac{\p}{\p x}+I\frac{\p}{\p y}\right)f_I(x+yI)=0\]
for all $x+yI\in \BB
\cap L_I$.
The class of slice regular functions was introduced in \cite{gentili-struppa} and we refer the reader to the monograph \cite{libroGSS} for an introduction to this topic.

The class of slice functions was introduced few years later by R. Ghiloni and A. Perotti \cite{ghiloniperotti} in a more general setting than the one of this work.  In this paper we adopt as definition of slice functions the following one. Given a function $f:\p\BB\to\HH$ we say that $f$ is a \emph{slice function} if, for any $J, K\in\SS$, $J\neq K$, 
and any $x+yI\in \p\B$, it holds
\begin{equation}\label{repfor}
 f(x+yI)=\big((J-K)^{-1}J+I(J-K)^{-1}\big)f(x+yJ)-\big((J-K)^{-1}K+I(J-K)^{-1}\big)f(x+yK).
\end{equation}
For the particular choice of $K=-J$ the above formula reduces to
\begin{equation}\label{repfor2}
\begin{aligned}
f(x+yI)&=\frac{1-IJ}{2}f(x+yJ)+\frac{1+IJ}{2}f(x-yJ).
\end{aligned}
\end{equation}
It is well-known that a slice regular function $f:\B\to \HH$ satisfies these representation formulas on the whole unit ball $\BB$; see \cite{libroGSS}.

We now recall the definition of slice $L^p$ spaces. These spaces were introduced by the second author in \cite{sarfatti-indiana}. We focus here only on the Hilbert case $p=2$ and the case $p=\infty$ since it is enough for the purposes of this work. Let us consider the space of $\Sigma$-measurable functions
$$
L^2(\p\B)=\left\{f:\p\B\to\HH: \|f\|^2_2=\int_{\p\B}|f|^2\,d\Sigma<\infty\right\}.
$$
As usual, the functions in $L^2(\p \B)$ that coincide $\Sigma$-almost everywhere are identified.\\
Then, the space of slice $L^2$ functions, which from now on we denote by $L^2_s(\p\B)$, is the closed subspace of $L^2(\p\B)$ consisting of slice functions. 
In \cite{sarfatti-indiana} it is proved that 
$$
L^2_s(\p\B)=\bigg\{f\in L^2(\p\B)\, : \,f(q)= \sum_{n\in \Z}q^n a_n \, \textrm{for $\Sigma$-almost every q}, \{a_n\}_{n\in \Z} \in \ell^2(\Z, \HH) \bigg\}.
$$
In particular, $ L^2_{s} (\p\B)$ endowed with the inner 
product  
\[ \Big \langle \sum_{n\in \Z}q^n a_n,\sum_{n\in \Z}q^n b_n\Big \rangle_{ L^2_{s} (\p\B)}:=\sum_{n\in \Z}\bar b_n a_n\]
is a quaternionic right Hilbert space and $\{q^n\}_{n\in\Z}$ is an orthonormal basis with respect to this inner product. 
Moreover, given any $I\in\s$, the inner product of $L^2_s(\p\B)$ has the integral representation
\begin{equation*}\label{sliceinnerproduct}
 \left<f,g\right>_{L^2_s(\p\B)}=\frac{1}{2\pi}\int_{0}^{2\pi}\overline{g(e^{I\theta})}f(e^{I\theta})d\theta.
\end{equation*}

\noindent Similarly, we define $L^\infty_s(\p\B)$ to be the space of $\Sigma$-measurable slice functions that are essentially bounded, that is, the space of slice functions such that
$$
\|f\|_{L^\infty_s(\p\B)}:=\ess\sup_{q\in\p\B}|f(q)|<\infty,
$$
where the essential supremum is taken with respect to the measure $\Sigma$,
$$
\ess\sup_{q\in\p\B} |f(q)|=\inf\{\lambda\geq 0:\Sigma(\{q\in\p\B:|f(q)|>\lambda\})=0\}.
$$

Before stating our results, we need a few more definitions. If $f(q)=\sum_{n\in\Z}q^na_n$ is a slice function on $\p\B$, then the \emph{conjugate} of $f$ is the defined as
\begin{equation}\label{regular-conjugate}
 f^c(q):=\sum_{n\in\Z}q^n\overline{a_n}.
\end{equation}
Morevover, we denote by $\widetilde{f}$ the function 
\begin{equation*}\label{function-tilde}
 \widetilde f(q):=f(\overline q).
\end{equation*}
In general, the pointwise product of two slice (or slice regular) functions is not a slice function, thus a suitable product must be considered, namely, the so-called \emph{slice} or $\ast$-\emph{product}. If $f(q)=\sum_{n\in\Z} q^n a_n$ and $g(q)=\sum_{n\in\Z} q^n b_n$ are two slice functions on $\p\B$, then
$$
f\ast g(q):=\sum_{n\in\Z}q^n\sum_{k\in\Z}a_k b_{n-k}.
$$
This product is related to the pointwise product by the formula
$$
f* g(q)=\left\{\begin{array}{l r}
0 & \text{if $f(q)=0 $}\\
f(q)g(T_{f^c}(q)) & \text{if $f(q)\neq 0 $\,,}
                \end{array}\right.
$$
where $T_{f^c}(q):= f(q)^{-1}q f(q)$. We refer the reader to \cite{libroGSS,sarfatti-indiana} for a proof of this fact.

We denote by $f^{-\ast}$ the inverse of $f$ with respect to the $\ast$-product. Clearly, the function $f^{-\ast}$ is not necessarily defined wherever $f$ is. We refer the reader to the introduction of Section \ref{prelim} for a precise definition of $f^{-\ast}$.

Similarly to the classical setting, it is natural to investigate the invariant subspaces of $L^2_s(\p\B)$ for the ($*$-)multiplier operator $$
M_qf=q\ast f.
$$
The $\ast$-product is clearly non-commutative; nonetheless, by direct computations, it follows
$$
M_qf=q\ast f= f\ast q=qf. 
$$
Then, our first result is a complete characterization of the doubly invariant subspaces, that is, of those closed subspaces $\clK\subseteq L^2_s(\p\B)$ such that $M_q\clK=\clK$.

\begin{teo}\label{thm-doubly-invariant}
	A closed subspace $\clK \subseteq L^2_s(\partial \B)$ is a doubly invariant subspace if and only if there exists a unique function $\varphi\in L^2_s(\p\B)$ such that $\clK=\varphi*L_s^2(\partial\B)$ and $\varphi$ satisfies $\widetilde{\varphi}*\varphi^c=\widetilde{\varphi}$. 	\end{teo}

\noindent Notice that $\varphi\in L^2_s(\p\B)$ satisfies $\widetilde{\varphi}*\varphi^c=\widetilde{\varphi}$ if and only if $\varphi$ satisfies the conditions

$$
\left\{\begin{array}{l r}
\varphi\ast \varphi=\varphi \\
\widetilde \varphi^c=\varphi. 
                \end{array}\right.
$$
In fact, $\widetilde{\varphi}*\varphi^c=\widetilde{\varphi}$ if and only if
\[ \varphi*\widetilde{\varphi}^c=(\widetilde{\varphi}*\varphi^c)^c=\widetilde\varphi^c
\quad \text{ and } \quad \varphi*\widetilde\varphi^c=\widetilde{\widetilde{\varphi}*\varphi^c}=\varphi.\]
As we will see, the properties of the function $\varphi$ are the ones needed to guarantee that the multiplier operator $M_\varphi f:=\varphi\ast f$ is a projection, that is, $M_\varphi^2=M_\varphi$, and is self-adjoint on $L^2_s(\p\B)$.

The characterization of the doubly invariant subspaces is more explicit thanks to the following result. 
\begin{teo}\label{thm-idem}
	Let $f\in L^2_s(\p\B)$ be such that $f*f=f$. Then on each sphere $e^{t\s}\subseteq\p\B$ the function $f$ behaves as follows. If $f|_{e^{t\s}}$ denotes the restriction of $f$ to the sphere $e^{t\s}$, then $f|_{e^{t\s}}$ is either constant $0$ or $1$, 
	or there exist $J=J(t),K=K(t) \in\s, J\neq K$ such that 
	$$
	f(e^{tI})=-\big((J-K)^{-1}K+I(J-K)^{-1}\big)
	$$
for any $e^{tI}\in e^{t\s}$. Moreover, if $f$ satisfies also $\widetilde f^c=f$, then $K=-J$.

\end{teo}

The converse of Theorem \ref{thm-idem} holds as well as can be easily checked by direct computations. Therefore, we have a characterization of slice $L^2$ idempotent functions. In particular, these functions are necessarily bounded. We point out that from this characterization we also deduce that the function $\varphi$ in Theorem \ref{thm-doubly-invariant} is not a characteristic function as in the classical setting. We refer the reader to the examples in Section \ref{prelim}.

The proof of Theorem \ref{thm-idem} exploits the Representation Formula \eqref{repfor}. The simplified Representation Formula \eqref{repfor2} could be used as well and the characterization theorem should be reformulated accordingly; see Theorem \ref{thm-idem-2}.

We now focus on the simply invariant subspaces of $L^2_s(\p\B)$, that is, those closed subspaces $\clK\subseteq L^2_s(\p\B)$ such that $M_q\clK\subsetneq \clK$. Let $H^2(\p\B)$ denote the Hardy space 
$$
H^2(\p\B):=\left\{f\in L_s^2(\p\B)\,: \,f(q)= \sum_{n\in \N}q^n a_n \,,\{a_n\}_{n\in \N} \in \ell^2(\N, \HH) \right\}.
$$
Then, $H^2(\p\B)$ is a closed subspace of $L^2_s(\p\B)$ which can be identified with the quaternionic Hardy space of slice regular functions $H^2(\B)$ ; see, e.g., \cite{libroMilanesi, deFGS}. When we write $f\in H^2(\p\B)$ it will always be implicit that $f$ actually is the boundary value function of a slice regular regular function defined on the whole unit ball $\B$. Having this in mind, if $f\in H^2(\p\B)$, then the expression $f(q)$ for $|q|<1$ is meaningful.

We prove the following result.
\begin{teo}\label{thm-simply-invariant}
	A closed subspace $\clK\subseteq L^2_s(\p\B)$ is a simply invariant subspace if and only if $\clK=\varphi * H^2(\partial \B)$ for some function $\varphi \in L^{\infty}_s(\partial \B)$ such that $|\varphi|=1$ $\Sigma$-almost everywhere on $\partial \B$. 
	
	Moreover, such a function is unique up to a unitary constant if the following sense: if $\clK=\varphi_1\ast H^2(\p\B)= \varphi_2\ast H^2(\p\B)$, then $\varphi_2=\varphi_1\ast u$ for some $u\in\HH$ such that $|u|=1$.
\end{teo}

A corollary of the latter result is an analogue of Beurling's theorem. We point out that hypercomplex versions of Beurling's theorem already appeared in \cite{milanesi-krein,libroMilanesi,milanesi-beurling}. Nonetheless we state it here for the sake of completeness. 

Let $H^\infty(\B)$ denotes the space of bounded slice regular functions on the unit ball. Then, each function $f\in H^\infty(\B)$ admits a boundary value function, which we still denote by $f$. Hence, the space $H^\infty(\B)$ can be identified with a closed subspace of $L^\infty_s(\p\B)$ which we denote by $H^\infty(\p\B)$. As in the $H^2$ setting, if $f\in\ H^\infty(\p\B)$, then in the expression $f(q)$ for $|q|<1$ it is implicit that we are considering the slice regular extension of $f$ to the whole unit ball $\B$.

 A function $\varphi\in H^\infty(\p\B)$ is an {\em inner function} for $H^2(\p\B)$ if $|\varphi(q)|\leq 1$ on $\B$ and $|\varphi(q)|=1$ $\Sigma$-almost everywhere on $\p\B$. Then, the following theorem holds.
 
\begin{teo}\label{thm-Beurling}
 Let $\clK$ be a subspace of $H^2(\partial \B)$. Then $\clK$ is a closed invariant subspace if and only if $\clK=\varphi * H^2(\partial \B)$ for some inner function $\varphi \in H^\infty(\p\B)$.
	
	Moreover, such a function is unique up to a unitary constant in the following sense: if $\clK=\varphi_1\ast H^2(\p\B)= \varphi_2\ast H^2(\p\B)$, then $\varphi_2=\varphi_1\ast u$ for some $u\in\HH$ such that $|u|=1$.
\end{teo}
Finally, we deduce from Beurling's theorem the inner-outer factorization of functions in $H^2(\p\B)$. A function $g\in H^2(\p\B)$ is an {\em outer function} for $H^2(\p\BB)$ if for any $f\in H^2(\p\B)$ such that $|g(q)|=|f(q)|$ for $\Sigma$-almost every $q\in\p\B$, it holds $|g(q)|\geq |f(q)|$ for any $q\in\B$ \cite[Definition 5.1]{deFGS}.
Then, our factorization result is the following.

\begin{teo}\label{thm-factorization}
Let $f\in H^2(\p\B)$, $f\not\equiv 0$. Then $f$ has a factorization $f=\varphi*g$ where $\varphi$ is inner and $g$ is outer.

Moreover, this factorization is unique up to a unitary constant in the following sense: if $f=\varphi\ast g=\varphi_1\ast g_1$, then $\varphi_1=\varphi\ast\lambda$ and $g_1=\overline{\lambda}\ast g$ for some $\lambda\in\HH$ such that $|\lambda|=1$.
\end{teo}

In order to prove the factorization theorem we also provide a characterization of outer functions in terms of cyclicity. A function $g$ is {\em cyclic}  for $H^2(\p\BB)$ if and only if  
\begin{equation}\label{outer}
 E_g:=\overline{\Span\left\{q^n\ast g, n\geq 0\right\}}=H^2(\p\BB).
\end{equation}
Thus, we prove that a  function $g$ is outer for $H^2(\p\B)$ if and only if it is cyclic; see Theorem \ref{prop-outer}. 

We point out that we are working with quaternionic right Hilbert spaces, thus the left-hand side of \eqref{outer} denotes the closure in $H^2(\p\BB)$ of elements of the form 
\[\sum_{n=0}^M (q^n\ast g)\alpha_n=\sum_{n=0}^M (g\ast q^n)\alpha_n=g*p_M\] 
where $p_M$ is a quaternionic polynomial.

The paper is organized as follows. In Section \ref{prelim} we recall some basic facts about slice functions, we prove some properties we will need in the remaining of the paper and we prove Theorem \ref{thm-idem}. In Section \ref{sec-multiplier-operators} we study on $L^2_s(\p\B)$ the multiplier operator $M_\varphi$ associated to a generic slice function $\varphi$, whereas in Section \ref{sec-invariant-subspaces} we prove the characterization of invariant subspaces and the factorization result.

\section{Preliminary results and slice idempotent functions}\label{prelim}
In this section we recall some notation and basic properties of slice functions, we prove some auxiliary results we will need later and we conclude the section with the proof of Theorem \ref{thm-idem}.

We already defined in Formula \eqref{regular-conjugate} the conjugate of a slice function $f$. Let us now recall the definition of symmetrization. If $f(q)=\sum_{n\in \Z}q^n a_n$ is a slice function on $\p\B$, the {\em symmetrization} of $f$ is defined as
\[  f^s(q):=f^c* f (q)=f* f^c (q).\] 
As in the case of slice regular functions, it holds  $(f* g)^c=g^c* f^c$; see \cite{sarfatti-indiana}. The reciprocal $f^{-*}$ of $f(q)=\sum_{n\in \Z}q^na_n$ with respect to the $*$-product is then given by 

\[f^{-*}(q)=(f^s(q))^{-1}f^c(q).\]

The function $f^{-*}$ is defined on $\p\B \setminus \{q\in \p\B \ | \ f^s(q)=0\}$ and $f* f^{-*}=f^{-*}* f =1$. We recall that $f^{c}, f^s$ and $f^{-\ast}$ are slice (slice regular) functions whenever $f$ is a slice (slice regular) function. 

Notice that the $L^2$-norm of a function $f\in L^2_s(\p\B)$ depends only on the moduli of the coefficients of its power series expansion. Hence $f\in L^2_s(\p\B)$ if and only if its conjugate $f^c$ does, and $\|f\|^2_{L^2_s(\p\B)}=\|f^c\|^2_{L^2_s(\p\B)}$. 

The analogous property holds for $L^\infty_s(\p\B)$, that is, a function $\varphi$ belongs to $L^{\infty}_s(\p\B)$ if and only if its conjugate $\varphi^c$ does and the two norms coincide, $\|\varphi\|_{L^{\infty}_s(\p\B)}=\|\varphi^c\|_{L^{\infty}_s(\p\B)}$; see \cite{sarfatti-indiana}. Moreover we have the following proposition.
\begin{pro}\label{modulophic}
Let $\varphi\in L^\infty_s(\p\B)$. Then, $|\varphi|=1$ $\Sigma$-almost everywhere on $\p\B$ if and only if $|\varphi^c|=1$ $\Sigma$-almost everywhere on $\p\B$.
\end{pro}
\begin{proof}
 The proof follows the same lines of the proof \cite[Proposition 5]{BlochDGS} for slice regular functions, thus we do not include the details here. See also \cite[Theorem 3.7]{gps}.
 \end{proof}

In order to prove Theorem \ref{thm-simply-invariant} we need a result on the $\Sigma$-measure of the zero set of the symmetrization of a a bounded slice function. We refer the reader to Formula \eqref{Sigma} for the definition of the measure $\Sigma$.

Let $Z_{\varphi}=\{q\in\p\B: \varphi(q)=0\}$ be the zero set of the function $\varphi$. Then, as for slice regular functions (see, e.g., \cite[Proposition 3.9]{libroGSS}),
\begin{equation}\label{zerosetsim}
Z_{\varphi^s}=\bigcup_{e^{tI}\in Z_\varphi} e^{t\SS}.
\end{equation}

We recall also that if $\varphi:\B\to\HH$ is a slice function, then the map
\begin{equation}\label{op-T-phi}
T_{\varphi^c}(q)= \varphi(q)^{-1}q\varphi(q)
\end{equation}
is a bijection from $\B\backslash Z_{\varphi^s}$ to itself; its inverse, as in the case of slice regular functions, is the map $T_{\varphi}$ (see \cite[Proposition 5.32]{libroGSS}). Moreover, if $\varphi:\p\B\to\HH$ and $|\varphi|=1$ $\Sigma$-almost everywhere on $\p\B$, then $T_{\varphi^c}$ is a bijection from $\p\B\backslash Z_{\varphi^s}$ to itself.

The following result holds.
\begin{pro}\label{zerisimInner}
 	Let $\varphi \in L^{\infty}_s(\B)$ be such that $|\varphi|=1$ $\Sigma$-almost everywhere on $\p\B$. Then, $Z_{\varphi^s}$ has vanishing $\Sigma$-measure. Moreover it holds that $|\varphi^s|=1$ $\Sigma$-almost everywhere on $\p\B$ as well. 
\end{pro}

\begin{proof}
From the Representation Formula \eqref{repfor} we deduce that if $\varphi$ vanishes at $e^{tI}$, then either $\varphi$ vanishes identically on the sphere $e^{t\s}$, and we call such a sphere a {\em spherical zero}, or it does not have any other zero in $e^{t\s}$. Let $A$ be the union of spherical zeros of $\varphi$ and $B$ the union of spheres where $\varphi$ vanishes only at a point.
Since $|\varphi(q)|=1$ for $\Sigma$-almost every $q$, we immediately get that $\Sigma(A)=0$.
Consider now a sphere $e^{t\s}$ in $B$ and let $e^{tI_0}\in e^{t\s}$ be the only point where $\varphi$ vanishes. Then we have that $\varphi(e^{-tI_0})=c$ for some quaternion $c\neq 0$ and hence, thanks to the Representation Formula \eqref{repfor2}, that, for any other $J\in \s$,
\[\varphi(e^{tJ})=\frac{1-JI_0}{2}\varphi(e^{tI_0})+\frac{1+JI_0}{2}\varphi(e^{-tI_0})=\frac{1+JI_0}{2}c.\]
In particular $|\varphi(e^{tJ})|$ is different from $1$ for $\sigma$-almost every $J\in \s$. Thus $dt(\{t : e^{t\s}\subseteq B\})=0$ and $\Sigma(B)=0$ as well.

To prove the last part of the statement, we use the fact that  $T_{\varphi^c}$ is a bijection from $\p\B\backslash Z_{\varphi^s}$ to itself as we remarked above. Since $Z_{\varphi^s}$ has $\Sigma$-measure zero, for $\Sigma$-almost every $q\in\p\B$, we get
\[|\varphi^s(q)|=|\varphi(q)||\varphi^c(T_{\varphi^c}(q))|.\]
Now, recalling that $|\varphi|=1$ $\Sigma$-almost everywhere on $\p\B$ if and only if the same holds true for $\varphi^c$, we conclude that $|\varphi^s|=1$ $\Sigma$-almost everywhere on $\p\B$ as we wished to show.
\end{proof}
 Also the symmetrization of a function in $H^2(\p\B)$ cannot vanish on a set with positive $\Sigma$-measure.
 \begin{pro}\label{zerisimH2}
 	Let $f\in H^2(\p\B)$, $f\not\equiv 0$. Then $Z_{f^s}$ has vanishing $\Sigma$-measure.
 \end{pro}
 \begin{proof}
  
  Set $Z_{f^s}^{I_0}:=\{e^{t{I_0}} \in\partial \B:  t\in (0,\pi), e^{t\s}\subseteq Z_{f^s}\}$. The fact that $f$ is in $H^2(\partial \B)$ and does not vanish identically implies (see \cite[Propositions 3.13 and 4.4]{deFGS}) that $f^s\in H^1(\partial \B)$ and, on each slice $\partial \B\cap L_{I_0}$, it vanishes on a zero $dt$-measure set, thus $\int_{Z_{f^s}^{I_0}}dt=0$. Notice that $Z_{f^s}$ is symmetric with 
  respect to the real axis and hence it can be decomposed as 
  \begin{align*}
  (Z_{f^s}\cap \rr)\cup (Z_{f^s}^{I_0}\times \s)&=(Z_{f^s}\cap \rr)\cup \{e^{t\s}\subseteq\partial \B : e^{tI_0}\in Z_{f^s}^{I_0}\}\\
  &=(Z_{f^s}\cap \rr)\cup \{e^{t\s}\subseteq\partial \B :  t\in (0,\pi), \ e^{t\s}\subseteq Z_{f^s}\}=Z_{f^s}.
  \end{align*}
  Since $(Z_{f^s}\cap \rr)\subseteq\{-1,1\}$ we conclude that
  \[\Sigma(Z_{f^s})=\int_{Z_{f^s}}d\Sigma(e^{tI})=\int_{\s} \int_{Z_{f^s}^{I_0}}dt d\sigma(I)=\int_{Z_{f^s}^{I_0}}dt\int_{\s}  d\sigma(I)=0.\]
 \end{proof}

From the above proposition we deduce the following result.
\begin{lem}\label{modulo}
	Let $\varphi\in L_s^{\infty}(\partial\B)$ and denote by $\widetilde{\varphi}$ the slice function defined by $\widetilde{\varphi}(q)=\varphi(\bar q)$.
	Then $|\varphi|=1$ $\Sigma$-almost everywhere on $\partial \B$ if and only if  $\widetilde{\varphi}*\varphi^c=\varphi^c*\widetilde{\varphi}=1$ $\Sigma$-almost everywhere on $\partial \B$.

\end{lem}
\begin{proof}
	Consider the power series expansion of $\varphi$, $\varphi(e^{tI})=\sum_{n\in\Z}e^{ntI}a_n$ holding $\Sigma$-almost everywhere on $\p\B$.
Then,
\begin{align*}
|\varphi(e^{tI})|^2&=\overline{\sum_{n\in\Z}e^{ntI}a_n}\sum_{m\in\Z}e^{mtI}a_m=\sum_{n\in\Z}\overline{a_n}e^{-ntI}\sum_{m\in\Z}e^{mtI}a_m=
\sum_{n,m\in\Z}\overline{a_n}e^{(m-n)tI}a_m\,.
\end{align*}
On the other hand,
\begin{align*}
\widetilde{\varphi}*\varphi^c(e^{tI})&=\sum_{n\in\Z}e^{ntI}\sum_{m\in\Z}a_{-m}\overline{a_{n-m}}
=\sum_{n,m\in\Z}e^{(n-m)tI}a_{m}\overline{a_{n}}\,,
\end{align*}
where the last equality is just a relabeling of the indexes.
Recalling the equality $\RRe(ab)=\RRe(ba)$ for any $a,b\in\HH$, we get that $\Sigma$-almost everywhere it holds
\begin{align}\label{modulophi}
|\varphi(e^{-tI})|^2&
=\RRe(\widetilde{\varphi}*\varphi^c(e^{tI})).
\end{align}
Analogously, it is possible to prove that\[|\varphi^c(e^{tI})|^2=\RRe(\varphi^c*\widetilde{\varphi}(e^{tI})). \]
Now, if $\widetilde{\varphi}*\varphi^c(q)=\varphi^c*\widetilde{\varphi}(q)=1$ for $\Sigma$-almost every $q\in\p\B$, then $|\varphi(q)|=|\varphi^c(q)|=1$ for $\Sigma$-almost every $q$ as well.

In the other direction, if $|\varphi(q)|=1$ for $\Sigma$-almost every $q\in\p\B$, then also $|\widetilde{\varphi}(q)|=1$ and $|\varphi^c(q)|=1$ for $\Sigma$-almost every $q$ (since $q\mapsto \bar q$ is a diffeomorphism of $\p\B$ to itself and thanks to Proposition \ref{modulophic}). In particular, thanks to Proposition \ref{zerisimInner}, the zero set $Z_{\widetilde\varphi^s}$ of $\widetilde\varphi^s$ has measure zero and the map $T_{\widetilde\varphi^c}(q)=\widetilde{\varphi}(q)^{-1}q\widetilde{\varphi}(q)$ is a bijection of $\partial\mathbb B\backslash Z_{\widetilde\varphi^s}$ to itself, hence, for $\Sigma$-almost every $q$, it holds
\[
|\widetilde{\varphi}*\varphi^c(q)|=|\widetilde{\varphi}(q)||\varphi^c((\widetilde{\varphi}(q))^{-1}q\widetilde{\varphi}(q))|=1.
\]
Thus,  $\Sigma$-almost everywhere on $\p\B$, using also \eqref{modulophi}, we have that 
\[|\widetilde{\varphi}*\varphi^c(q)|=1=|\varphi(\bar q)|^2=\RRe(\widetilde{\varphi}*\varphi^c(q)),\]
which implies that $\widetilde{\varphi}*\varphi^c(q)$ is a (positive) real number and hence equal to its modulus $\Sigma$-almost everywhere on $\p\B$, that is,
\[\widetilde{\varphi}*\varphi^c(q)=|\widetilde{\varphi}*\varphi^c(q)|=1.\]
Similarly we also obtain $\varphi^c\ast\widetilde\varphi=1$ and the proof is complete.

\end{proof}
\begin{oss}
The previous result provides also a characterization for the inner functions of $H^2(\p\B)$. 
\end{oss}

We conclude the section proving Theorem \ref{thm-idem}. Before actually proving the theorem we collect some examples of idempotent functions to show that the class of slice idempotent functions is not trivial. 

\begin{itemize}
\item  A first example of an idempotent function is the characteristic function of  a set which is symmetric with respect to the real axis. Thus, let $E\subseteq \p\B$ be such a circular set. Then \[\chi_E*\chi_E(q)=\left\{\begin{array}{l r}
0 & \text{if $\chi_E(q)=0$}\,;\\
\chi_E(q)\chi_E(q) 	& \text{if $\chi_E(q)=1$\,,} 
\end{array} \right.\]
i.e., $\chi_E*\chi_E(q)=\chi_E(q)$. In this case, $\chi_E$ is constant on every sphere, either equal to $1$ or to $0$.
\item A second example is peculiarly quaternionic. Let $J$ be any imaginary unit, and consider the slice constant function (defined outside the real axis) $\ell(e^{tI})=\frac{1-IJ	}{2}$. Then, by direct computation it is possible to prove that $\ell*\ell(e^{tI})=\ell(e^{tI})$ for any $e^{tI}\in \p\B$.
The function $\ell$ can be interpreted as the slice extension of the characteristic function of the semi-circle $\{e^{tJ}: t\in (0,\pi)\}$ in $L_J$. 

A property that is worth mentioning is that, since these functions have exactly one zero on each sphere, then their symmetrization are identically vanishing.

\item A more complicated example

is given by
\[f(e^{tI})=\frac{1+I (\cos(t)i+\sin(t)j)}{2}. \] 
Again, by direct computation it is possible to prove that $f*f=f$.
As in the previous example, on each sphere $e^{t\s}$ contained in $\p\B$, the function $f$ takes at exactly one point the value $0$ and at its conjugate the value $1$. This time the imaginary unit $\cos(t)i+\sin(t)j$ of the zero of $f$ depends on the sphere. This example was suggested to us by A. Altavilla. Similar functions are studied for different purposes in \cite{adF}.
\end{itemize}

\begin{proof}[Proof of Theorem \ref{thm-idem}]
It is a simple matter of computations to verify that a function $f$ of the described form satisfies $f\ast f=f$. We now prove that all  idempotent functions are of such a form. From \eqref{zerosetsim} we know that if $f $ vanishes at a point $e^{tI}$, then $f^s|_{e^{t\s}}\equiv 0$. Thus, let us consider first a sphere where $f^s|_{e^{t\s}}\neq 0$. Then, equation $f*f=f$ implies that, for any $I\in \s$,
\[
f^s(e^{tI})=f^s*f(e^{tI})=f^s(e^{tI})f(e^{tI}),
\] 
hence $f(e^{tI})=1$ for any $I\in \s$.
Secondly, consider now a sphere $e^{t\s}$ such that  $f^s|_{e^{t\s}}\equiv 0$. In this case, since $f$ satisfies \eqref{repfor}, there are two possibilities: either $f|_{e^{t\s}}\equiv 0$ as well or $f|_{e^{t\s}}$ has an isolated zero $q_0=e^{tJ}$. If $f$ is constantly equal to zero  we are done; alternatively, consider $q\neq q_0$. Then, equation $f*f=f$ implies that
\[
f(q)=f(q)f(f(q)^{-1}q f(q))= f(q)f(T_{f^c}(q)),
\] 
i.e., $f(T_{f^c}(q))=1$ for every $q\neq q_0$. Recalling that the transformation $T_{f^c}$ maps each sphere $e^{t\s}$ into itself, we just proved that if $f$ has an isolated zero $q_0=e^{tJ}$ on the sphere $e^{t\s}$, then $f$ assumes value $1$ in one point of the same sphere, namely, the point $T_{f^{c}}(q)=e^{tK}$ for some $K\in\s$. We remark that $e^{tK}$ is the only point of the sphere $e^{t\s}$ where $f$ assumes value $1$,
in fact, if $f(q)=1$ for some other $q\in e^{t\s}$, then the representation formula would imply $f|_{e^{t\s}}\equiv1$. This is not the case since $f(q_0)=0$.
We also remark that $e^{tK}$ is, in particular, the isolated zero on $e^{t\s}$ of $f^c$, the conjugate of $f$. This follows from the fact that, for every $q\in e^{t\s}, q\neq q_0$, it holds
$$
0=f^s(q)=f(q) f^c(T_{f^c}(q))=f(q) f^c(e^{tK}).
$$
Similarly, it is possible to show that $f^c(e^{tJ})=1$.
Finally, using Formula \eqref{repfor}, we can represent the functions $f$ and $f^c$ on the sphere $e^{t\s}$ with respect to the points $e^{tJ}$ and $e^{tK}$, obtaining that
$$
f(e^{tI})=-\big((J-K)^{-1}K+I(J-K)^{-1}\big) 
$$
and 
$$
f^c(e^{tI})=((J-K)^{-1}J+I(J-K)^{-1})
$$
for any $I\in \s$.
In particular, if $f=\widetilde{f}^c$, necessarily $K=-J$.

\end{proof}
We proved Theorem \ref{thm-idem} using the Representation Formula \eqref{repfor}. If we want to use the simplified formula \eqref{repfor2}, we can reformulate the theorem as follows.
\begin{teo}\label{thm-idem-2}
Let $f\in L^2_s(\p\B)$ be such that $f*f=f$. Then on each sphere $e^{t\s}\subseteq\p\B$ the function $f$ behaves as follows. If $f|_{e^{t\s}}$ denotes the restriction of $f$ to the sphere $e^{t\s}$, then, $f|_{e^{t\s}}$ is either constant $0$ or $1$, or there exists $J=J(t)\in\s$ such that 
	\begin{equation}\label{idem-repfor}
	f(e^{tI})=\frac{1+IJ}{2}f(e^{-tJ}).
	\end{equation}
Moreover, 
$f(e^{-tJ})=1+yK$ for some $y\in \RR$ and $K\in\s$ such that $K\,\bot\, J$ (with respect to the scalar product of\, $\RR^3$). In particular, if $\widetilde{f}^c=f$, then $f(e^{-tJ})=1$.
\end{teo}
\begin{proof}
 The fact that on each sphere $f$ is constant (either $0$ or 1) or satisfies \eqref{idem-repfor} follows as in the proof of Theorem \ref{thm-idem}. In particular, equation \eqref{idem-repfor} holds if $f$ has an isolated zero $e^{tJ}$ on the sphere $e^{t\s}$. Thus, it only remains to investigate the term $f(e^{-tJ})$. In order to ease the notation in the computations, we set $a:= f(e^{-tJ})$. Then, from equation $f\ast f=f$ and \eqref{idem-repfor} we get
 $$
 a=af(a^{-1}e^{-tJ}a)
 $$
 i.e., $f(a^{-1}e^{-tJ}a)=1$. Computing $f(a^{-1}e^{-tJ}a)$ explicitly we obtain the equation
 $$
 \frac{1-(a^{-1}Ja)J}{2}a=1,
 $$
that holds if and only if
\begin{equation}\label{a-comp}
aJ+Ja=2J.
\end{equation}
In particular, we observe that $\RRe(aJ+Ja)=0$. Assuming now that $a$ is of the form $a=x+yK$ for some $K\in\s$ and $x,y\in\RR$, equation \eqref{a-comp} holds true if and only if
$$
2xJ+y(KJ+JK)=2J.
$$
It is well-known that $KJ=-\left<K,J\right>+K\times J$ where $\left<\cdot,\cdot\right>$ and $\times$ denote the standard scalar and cross product in $\RR^3$ respectively. Thus, equation \eqref{a-comp} becomes
$$
2xJ-2y\left<K,J\right>=2J
$$
and this last equality is satisfied if and only if $x=1$ and $y=0$ or
$x=1$, $y\neq0$ and $K\,\bot\, J$, 
that is, if and only if 
$a=f(e^{-tJ})=1+yK$ for $y\in\rr$ and $K\bot J$, as we wished to show. 
In particular, if $f=\widetilde{f}^c$, with analogous computations it is possible to prove that $a=f(e^{-tJ})=1$.
\end{proof}

\section{Multiplier operators on $L^2_s(\p\B)$}\label{sec-multiplier-operators}

In this short section we extend to the setting of slice $L^2$ functions some concerning multiplier operators on $H^2(\p\B)$ proved in \cite{milanesi-nevanlinna}. The results in this section are not hard to prove and follow from some easy observations. Nonetheless, for the reader's convenience and for future reference, we state the results here as propositions.  

 Namely, we study the boundedness of $M_\varphi$, the multiplier operator associated to a generic slice measurable function $\varphi$, on $L^2_s(\p\B)$ and we explicitly write $M^\dagger_\varphi$, the adjoint operator of $M_\varphi$. The following result is standard and we do not include the proof.

\begin{pro}
	Let $\varphi$ be a measurable slice function on $\partial \B$. Then, the multiplier operator $M_\varphi:L^2_s(\partial \B) \to L^2_s(\partial \B), f \mapsto \varphi*f$, is a bounded linear operator if and only if $\varphi \in L_s^{\infty}(\partial \B)$. 
\end{pro}	

We would like to refine the previous result and prove that the operator $M_{\varphi}: L^2_s(\partial \B)\to L^2_s(\partial \B)$ is an isometry if we assume that $|\varphi|=1$ $\Sigma$-almost everywhere on $\partial \B$. This fact is a consequence of the following proposition.
\begin{pro}\label{adjoint}
 Let us consider the multiplier operator $M_\varphi:L^2_s(\partial \B)\to L^2_s(\partial \B)$ where $\varphi\in L^\infty_s(\partial \B)$. Then, the adjoint operator $M^\dagger_\varphi: L^2_s(\p\B)\to L^2_s(\p \B)$ is the multiplier operator associated to $\widetilde\varphi^c$, that is, $M^\dagger_\varphi=M_{\widetilde \varphi^c}$.
\end{pro}
\begin{proof}
 Consider the power series expansion of $\varphi(q)=\sum_{n\in\Z}q^n a_n$. Then, $\widetilde\varphi^c(q)=\sum_{n\in\Z}q^n\overline{a_{-n}}$. If $f(q)=\sum_{n\in\Z} q^nb_n$ is a function in $L^2_s(\p\B)$, by definition of adjoint operator, it holds
 \begin{align*}
  \big<M^\dagger_\varphi f, q^n\big>_{L^2_s(\p\B)}&= \big< f,M_\varphi q^n\big>_{L^2_s(\p\B)}= \sum_{k\in\Z}\overline{a_{k-n}}b_k.
 \end{align*}
Since $\{q^n\}_{n\in\Z}$ is a orthonormal basis for $L^2_s(\p\B)$ we conclude that
$$
M^\dagger_\varphi f(q)= \sum_{n\in\Z}q^n \sum_{k\in\Z}\overline{a_{k-n}}b_k= \sum_{n\in\Z}q^n\sum_{k\in\Z}\overline{a_{-k}}b_{n-k}=M_{\widetilde\varphi^c}f(q)
$$
 as we wished to show.
\end{proof}

\begin{coro}\label{isometria}
	Let $\varphi\in L_s^{\infty}(\partial\B)$. Then, $|\varphi|=1$ $\Sigma$-almost everywhere on $\partial \B$ if and only if its associated multiplier operator $M_\varphi:L_s^2(\partial \B) \to L_s^2(\partial \B)$ is a surjective isometry.
\end{coro}
\begin{proof}
On the one hand, if $|\varphi|=1$ Lemma \ref{modulo} guarantees that $\widetilde\varphi\ast\varphi^c=\varphi^c\ast\widetilde\varphi=1$, from which we also get
$$
(\widetilde\varphi\ast\varphi^c)^c=(\varphi^c\ast\widetilde\varphi)^c=1,
$$
that is,
$$
\widetilde\varphi^c\ast\varphi=\varphi\ast\widetilde\varphi^c=1.
$$
Thus, from the previous proposition we obtain $M^\dagger_\varphi M_\varphi=M_\varphi M^\dagger_\varphi= Id$, hence $M_\varphi$ is a unitary operator and, in particular, a surjective isometry.
On the other hand if $M_\varphi$ is a surjective isometry, then for any $m,n\in\Z$,
\[\langle q^m, q^n \rangle_{L^2_s(\p\B)}=\langle M_\varphi q^m, M_\varphi q^n \rangle_{L^2_s(\p\B)}=\langle  q^m, M_\varphi^\dagger M_\varphi q^n \rangle_{L^2_s(\p\B)}\]
namely $M_\varphi^\dagger M_\varphi=Id$. Thus $\widetilde{\varphi}^c*\varphi=1$ and, thanks to Lemma \ref{modulo}, we conclude.
\end{proof}

\section{Invariant subspaces and the inner-outer factorization}\label{sec-invariant-subspaces}
In this section we prove our main results.

\begin{proof}[Proof of Theorem \ref{thm-doubly-invariant}]
	Suppose first that $\varphi\in L^2_s(\p\B)$ is such that $\varphi\ast \varphi=\varphi$ and $\widetilde {\varphi^c}=\varphi$ and let us prove that $\varphi*L^2_s(\p\B)$ is a closed doubly invariant subspace. Consider the multiplier operator $M_\varphi$ associated with $\varphi$. Thanks to Theorem \ref{thm-idem}, $\varphi\in L^\infty_s(\p\B)$, hence the operator $M_\varphi:L^2_s(\p\B)\to L^2_s(\p\B)$ is bounded. The hypothesis on $\varphi$ clearly imply that $M_\varphi ^2=M_\varphi$ and $M^\dagger_\varphi=M_\varphi$. That is, $M_\varphi$ is a self-adjoint projection operator.  In particular, we get that $\varphi\ast L^2_s(\p\B)$, the range of $M_\varphi$, is a closed subspace of $L^2_s(\p\B)$. See, for instance, \cite{gmp}. 

	The double invariance of $\varphi\ast L^2_s(\p\B)$ is immediate since
	\[q*\varphi*L^2_s(\p\B)=\varphi*q*L^2_s(\p\B)= \varphi*L^2_s(\p\B).\]

	Consider now a closed subspace $\clK\subseteq L^2_s(\p\B)$ such that  $M_q \clK = \clK$. Let $\varphi=P_\clK(1)$ be the projection on $\clK$ of the constant function $1$. Then $1-\varphi \in \clK^\perp$. Since $\clK$ is doubly invariant, if $\varphi$ has power series expansion $\varphi(q)=\sum_{n \in \Z} q^{n}a_n$,  we get that, for any $k\in\Z$
	\begin{equation}\label{coeff2}
	0=\langle  M^k_q \varphi, 1-\varphi \rangle_{ L^2(\p\B)}=a_{-k}-\sum_{n\in\Z} \overline{a_n}a_{n-k}.
	\end{equation} 
Then Equation \eqref{coeff2} yields that, 	$\Sigma$-almost everywhere on $\p\B$, 
	\begin{equation*}\label{idem}
	\widetilde{\varphi}(q)=\sum_{k \in\Z} q^{k}a_{-k}=\sum_{k \in\Z} q^{k}\sum_{n\in \Z}\overline{a_n}a_{n-k}=\varphi^c*\widetilde{\varphi}(q). 
	\end{equation*}
It remains to show that $\clK=\varphi*L^2_s(\p\B)$.
Since $\varphi\in \clK$ and $\clK$ is doubly invariant, we get that $q^k\ast\varphi=\varphi\ast q^k\in \clK$ for any $k\in\Z$ and hence $\varphi\ast L^2_s(\p\B)\subseteq \clK$. 
Suppose now that $\varphi*L^2_s(\partial \B)\subsetneq \clK$. Then, since $\clK$ is a closed subspace, there exists $f\in \clK$ which is orthogonal to $\varphi*L_s^2(\partial \B)$. Then, for any $k\in \Z$,
\[0=
\langle  f, M_\varphi q^k  \rangle_{ L_s^2(\p\B)}=\langle M_\varphi^\dagger f, q^k  \rangle_{ L_s^2(\p\B)}=\langle \varphi*f, q^k  \rangle_{ L_s^2(\p\B)}\]
since $M_\varphi$ is self-adjoint. Thus, we get that $\varphi\ast f=0$. Moreover, for any $k\in\Z$,
\begin{align*}
 \langle (1-\varphi)\ast f, q^k\rangle_{ L_s^2(\p\B)}&=\langle M_{1-\varphi}f,q^{k}\rangle_{ L_s^2(\p\B)}=\langle f, M^\dagger_{1-\varphi} q^k\rangle_{ L_s^2(\p\B)}=\langle f, (1-\varphi)\ast q^k\rangle_{ L_s^2(\p\B)}\\
 &=\langle f, q^k\ast(1-\varphi)\rangle_{ L_s^2(\p\B)}=\langle q^{-k}\ast f,1-\varphi\rangle_{ L_s^2(\p\B)}=0
\end{align*}
where the last equality is due to the orthogonality of $1-\varphi$ and $M_{q^{-k}}(\clK)$.
Thus, $(1-\varphi)\ast f=0$. Hence, $f=1*f=\varphi*f+(1-\varphi)*f=0$ as we wished to show. 

To conclude the proof it remains to prove the uniqueness of the function $\varphi$. Assume that there exist two functions $\varphi_1,\varphi_2$ that satisfy $\widetilde\varphi_i\ast\varphi_i^c=\widetilde\varphi_i$, $i=1,2$, and $\mathcal K=\varphi_1\ast L^2_s(\p\B)=\varphi_2\ast L^2_s(\p\B)$. Then, we want to show that $\varphi_1=\varphi_2$. Since $M_{\varphi_i}$, $i=1,2$, is a projection and a self-adjoint operator, we get
$$
0=\langle 1-\varphi_1, \varphi_1\ast f\rangle_{L^2_s(\p\B)}=\langle 1-\varphi_2,\varphi_2\ast f\rangle_{L^2_s(\p\B)}
$$
for any $f\in L^2_s(\p\B)$. Thus, since $\varphi_1\ast L^2_s(\p\B)=\varphi_2\ast L^2_s(\p\B)$, we also get
$$
0=\langle 1-\varphi_1, \varphi_2\ast f\rangle_{L^2_s(\p\B)}=\langle 1-\varphi_2,\varphi_1\ast f\rangle_{L^2_s(\p\B)}
$$
for any $f\in L^2_s(\p \B)$. In particular,
$$
0=\langle 1-\varphi_1,\varphi_2\ast f\rangle_{L^2_s(\p\B)}=\langle \varphi_2-\varphi_2\ast\varphi_1, f\rangle_{L^2_s(\p\B)}
$$
for any $f\in L^2_s(\p\B)$. Hence, $\varphi_2=\varphi_2\ast\varphi_1$. Similarly we obtain also $\varphi_1=\varphi_1\ast\varphi_2$. Finally,
$$
\varphi_2=\widetilde\varphi_2^c=\widetilde\varphi_1^c\ast\widetilde\varphi^c_2=\varphi_1\ast\varphi_2=\varphi_1
$$
as we wished to show. This concludes the proof.
\end{proof}

\begin{proof}[Proof of Theorem \ref{thm-simply-invariant}]
Consider first $\clK=\varphi * H^2(\partial \B)$ for some measurable slice $\varphi$ such that $|\varphi|=1$ $\Sigma$-almost everywhere on $\partial \B$. Thanks to Corollary \ref{isometria} we have that $\clK$ is the image of a closed subspace through an isometry and hence it is a closed subspace itself.

Let us now show that $\clK$ is invariant but not doubly invariant with respect to the shift operator. It holds,
\[ M_q \clK=q*\clK=q*\varphi*H^2(\partial \B)=\varphi*q*H^2(\partial \B)\subseteq \varphi*H^2(\partial \B)=\clK.\]			
Moreover, $\varphi$ does not belong to $q*\varphi*H^2(\partial \B)$. In fact, suppose that $\varphi\in q\ast\varphi\ast H^2(\partial\mathbb B)$. Then, there exists $g\in H^2(\partial \mathbb B)$ such that $\varphi=q\ast\varphi\ast g=\varphi\ast q\ast g$. Hence, from Lemma \ref{modulo} we deduce $1=q\ast g$, that is, $g=q^{-1}\notin H^2(\partial\mathbb B)$.  Hence, $M_q \clK \subsetneq \clK$.
	
On the other hand, let $\clK$ be a closed subset of $L_s^2{(\partial\B)}$ such that  $ M_q \clK \subsetneq \clK$. Since $M_q$ is an isometry, we get that $M_q\clK$ is a closed subset of $\clK$. Hence, there exists $\varphi \in \clK$, with $\|\varphi\|_{L_s^2{(\partial \B)}}=1$, such that $\varphi$ is orthogonal to $ M_q\clK$ with respect to the $L_s^2(\partial \B)$ inner product. 
In particular, 
 such a $\varphi$ is orthogonal to $ M_q^k\varphi$ for any $k \geq 1$. Consider now the power expansion of $\varphi$, $\varphi(q)=\sum_{n\in \Z}q^{n}a_n$. Then, for any $k\geq 1$,
\[0=\langle  M_q^k \varphi, \varphi  \rangle_{ L_s^2(\p\B)}=\Big\langle \sum_{n\in \Z}q^{(n+k)}a_n, \sum_{n\in \Z}q^{n}a_n  \Big\rangle_{ L_s^2(\p\B)}=\sum_{n\in \Z}\overline{a_n}a_{n-k},\]
whereas, for any $k=-l<0$, 
\[0=\langle  \varphi,  M_q^l \varphi  \rangle_{ L_s^2(\p\B)}=\Big\langle \sum_{n\in \Z}q^{n}a_n, \sum_{n\in \Z}q^{(n+l)}a_n   \Big\rangle_{ L_s^2(\p\B)}=\sum_{n\in \Z}\overline{a_{n-l}}a_{n}=\sum_{n\in \Z}\overline{a_{n+k}}a_{n},\]
that is, for any $k\in \Z$, $k\neq 0$,
\begin{equation}\label{coeff}
\sum_{n\in \Z}\overline{a_{n}}a_{n-k}=0.
\end{equation} 
Then, for $\Sigma$-almost every $q\in \partial \B$,
\[\varphi^c*\widetilde{\varphi}(q)=\sum_{k\in \Z}q^{k}\sum_{n\in \Z}\overline{a_{n}}a_{n-k}=\sum_{n\in \Z}\overline{a_{n}}a_{n}=\|\varphi\|^2_{L_s^2(\partial \B)}=1.\]

Thanks to Lemma \ref{modulo} we get then that $|\varphi|=1$ $\Sigma$-almost everywhere on $\p\B$.

Let us conclude the proof showing that $\clK=\varphi*H^2(\partial\B)$.
On the one hand we have that the sequence $\{ M_q^n \varphi\}_{n\in \N}$ is an orthonormal sequence contained in $K$, but $\{M_q^n \varphi\}_{n\in \N}=\{q^n* \varphi\}_{n\in \N}=\{M_\varphi(q^n)\}_{n\in \N}$. Therefore this sequence is the image of the orthonormal basis $\{q^n\}_{n\in \N}$ of $H^2(\partial \B)$
through the isometry $M_\varphi$. Hence, $\varphi*H^2(\partial \B)=M_\varphi(H^2(\partial \B))$ is contained in $\clK$ and it is closed. On the other hand, suppose that $\varphi*H^2(\partial \B)\subsetneq \clK$ and 
consider $f\in \clK$ orthogonal to $\varphi*H^2(\partial \B)$. Then, for any $k\geq 0$, 
\[\langle  M_\varphi^\dagger f, q^k  \rangle_{ L_s^2(\p\B)}=\langle  f, M_\varphi q^k  \rangle_{ L_s^2(\p\B)}=0,\]
and, for any $k=-l<0$,
 \begin{align*}
 \langle  M_\varphi^\dagger f, q^{-l}  \rangle_{ L_s^2(\p\B)}&=\!\langle  f, M_\varphi q^{-l}  \rangle_{ L_s^2(\p\B)}=\!
 \langle  f, \varphi *q^{-l}  \rangle_{ L_s^2(\p\B)}
=\! \langle  f, q^{-l}\ast\varphi  \rangle_{ L_s^2(\p\B)}=\!
\langle M_q^lf,\varphi\rangle_{L_s^2(\p\B)}\!=\!0,
 \end{align*} 
where the last equality is due to the orthogonality of $\varphi$ and $ M_q^l (\clK)$.
Hence we have that $M_\varphi^\dagger f=0$ $\Sigma$-almost everywhere on $\partial\B$. Thus, $0=\varphi\ast\widetilde\varphi^c\ast f=f$ $\Sigma$-almost everywhere on $\p\B$ and, in particular, we conclude that $\mathcal K=\varphi\ast H^2(\p\B)$.

For the uniqueness, let $\varphi_1,\varphi_2$ be such that $\clK=\varphi_1*H^2(\partial \B)=\varphi_2*H^2(\partial \B)$. Then, exploiting  Lemma \ref{modulo}, we deduce that both 
$\widetilde\varphi^c_1\ast\varphi_2$ and $\widetilde\varphi^c_2\ast\varphi_1$ belong to $H^2(\p\B)$. Hence, also $(\widetilde\varphi^c_1\ast\varphi_2)^c=\varphi_2^c\ast\widetilde\varphi_1\in H^2(\p\B)$ and $(\widetilde\varphi^c_2\ast\varphi_1)^c=\varphi_1^c\ast\widetilde\varphi_2\in H^2(\p\B)$.

Using now the fact that for any slice function $f$ we have that $f\in H^2(\p\B)$ if and only if 
\[\widetilde f \in \widetilde H^2(\p \BB) :=  \Big\{ f\in L^2_s(\p\B) \,:\,f(q)= \sum_{n \in \N}q^{-n} a_{-n}, \,  \{a_{-n}\}_{n\in\N} \in {\ell^2(\N, \HH)} \Big\},\] 
we get that
$\widetilde{\varphi^c_1\ast\widetilde\varphi_2}=\widetilde\varphi^c_1\ast\varphi_2\in \widetilde H^2(\p\B)$.  So, $\widetilde\varphi^c_1\ast\varphi_2$ belongs both to $H^2(\p\B)$ and $\widetilde H^2(\p\B)$, thus $\widetilde\varphi^c_1\ast \varphi_2\equiv u$, with $u\in\HH$. Finally, applying again Lemma \ref{modulo}, we deduce that $\varphi_2=\varphi_1\ast u$ and, clearly, $|u|=1$ as we wished to show. 
\end{proof}

The proof of Theorem \ref{thm-Beurling}, Beurling's theorem for the quaternionic Hardy space, can be easily deduced from Theorem \ref{thm-simply-invariant}. 

We now exploit Theorem \ref{thm-Beurling} to prove the inner-outer factorization result. 
We recall that a function $g\in H^2(\p\B)$ is cyclic if 
$$
E_g:= \overline{\Span\big\{q^n\ast g, n\geq 0\big\}}= H^2(\p\B).
$$
It is easy to prove that $E_g$ is the smallest closed invariant subspace of $H^2(\p\B)$ containing $g$. Then, the following holds.
\begin{teo}\label{thm-fact-cyclic}
Let $f\in H^2(\p\B)$, $f\not\equiv 0$. Then $f$ has a factorization $f=\varphi*g$ where $\varphi$ is inner and $g$ is cyclic.

Moreover, this factorization is unique up to a unitary constant in the following sense: if $f=\varphi\ast g=\varphi_1\ast g_1$, then $\varphi_1=\varphi\ast u$ and $g_1=\overline{u}\ast g$ for some $u\in\HH$ such that $|u|=1$. 
\end{teo}
\begin{proof}
	Let $f\in H^2(\p\B)$ and let $E_f$ be the smallest closed invariant subspace of $H^2(\p\B)$ containing $f$. If $E_f$ coincides with $H^2(\p\B)$, then we are done. Otherwise, thanks to Theorem \ref{thm-Beurling} there exists an inner function $\varphi\in H^\infty(\p\B)$ such that $E_f=\varphi*H^2(\p\B)$, and hence there exists $g\in H^2(\p\B)$ such that $f=\varphi*g$. Let us show that $E_g=H^2(\p\B)$.
	Consider a function $h\in H^2(\p\B)$. Then $\varphi*h\in \varphi*H^2(\p\B)=E_f$. Therefore, there exists a sequence of polynomials $p_n$ such that $f*p_n$ converges in $H^2(\p\B)$ to $\varphi*h$. But
	\[\|f*p_n-\varphi*h\|_{L_s^2(\p\B)}=\|\varphi*g*p_n-\varphi*h\|_{L_s^2(\p\B)}=\|\varphi*(g*p_n-h)\|_{L_s^2(\p\B)},\]
	thus, recalling that if $\varphi$ is inner, then, thanks to Corollary \ref{isometria}, $M_\varphi$ is an isometry for $L_s^2(\p\B)$, we conclude that $g*p_n$ converges in $H^2(\p\B)$ to $h$, and hence that $h\in E_g$. 
	
The uniqueness follows from the uniqueness of the inner function identifying $E_f$:
if $f=\varphi_1*g_1$, with $\varphi_1$ inner and $g_1$ cyclic, then
\begin{align*}
\Span\big\{q^n\ast f, n\geq 0\big\}&=\Span\big\{q^n\ast \varphi_1*g_1, n\geq 0\big\}=\Span\big\{\varphi_1*q^n\ast g_1, n\geq 0\big\}\\&=\varphi_1*\Span\big\{q^n\ast g_1, n\geq 0\big\}.
\end{align*} 
Hence $E_f=\varphi_1*H^2(\p\B)$ which, thanks to Beurling's theorem, implies that $\varphi_1=\varphi*u$, with $|u|=1$.
Therefore $g_1=u^{-*}*g=\bar{u}*g$.
\end{proof}

Now, we want to prove that a function $g$ is cyclic if and only if it is outer. We refer the reader to the introduction for the definition of outer functions. Here we recall that if $g\in H^2(\p\B)$ is an outer function, then $g(q)\neq 0$ for every $q\in\B$. In fact, from \cite{deFGS} we know that $g$ admits the factorization $g(q)=h\ast b(q)$ where $h(q)\neq 0$ for any $q\in\B$ and $b$ is a Blaschke product. In particular, $|b(q)|=1$ for $\Sigma$-almost every $q\in\p\B$, thus $|g|=|h|$  $\Sigma$-almost everywhere on $\p\B$. Since $g$ is outer, it holds $|g(q)|\geq |h(q)|>0$ for any $q\in\B$, therefore $g$ cannot vanish in $\B$.

\begin{teo}\label{prop-outer}
 Let $g\in H^2(\p\B)$. Then, the following are equivalent:
 \begin{itemize}
 	\item[(i)] $g$ is cyclic, i.e., $E_g=H^2(\p\B)$;
 	\item[(ii)] for any $f\in H^2(\p\B)$ such that $f*g^{-*}\in L^2_s(\p\B)$, we have that $f*g^{-*}\in H^2(\p\B)$;
 	\item[(iii)] $g$ is outer, i.e., if $f\in H^2(\p\B)$ is such that $|f|=|g|$ \ $\Sigma$-almost everywhere on $\p\B$, then $|f(q)|\leq|g(q)|$ for any $q\in \B$.	
 \end{itemize}
\end{teo}
\begin{proof}
	Let us first show that $(i)$ is equivalent to $(ii)$. Suppose that $g$ is cyclic and let $f\in H^2(\p\B)$ be such that $f*g^{-*}$ is in $L^2_s(\p\B)$. We want to show that $\langle f*g^{-*},q^k \rangle_{ L^2_{s} (\p\B)}=0$ for any $k<0$. Consider the sequence of polynomials $\{p_n\}_{n\in\N}$ such that $g*p_n$ converges to the constant function $1$ in $H^2(\p\B)$. Notice that, for any $I \in \s$ and any $k\in\Z$, setting $\p\mathbb B_I= \p\mathbb B\cap L_I$,  we have
	\begin{align*}
|\langle f*g^{-*}-f*p_n,q^k \rangle_{ L^2_{s} (\p\B)}|&=\left|\frac{1}{2\pi}\int_0^{2\pi} e^{-tkI}(f*g^{-*}(e^{tI})-f*p_n(e^{tI}))dt\right|\\
&\leq\frac{1}{2\pi}\int_0^{2\pi}\left| f*g^{-*}*(1-g*p_n(e^{tI}))\right|dt\\
&=\frac{1}{2\pi}\int_0^{2\pi}\chi_{\{\p\B_I \setminus Z_{f*g^{-*}}\}}(e^{tI})\left| f*g^{-*}(e^{tI})\right|\left|1-g*p_n(T_{(f*g^{-*})^c}(e^{tI}))\right|dt\\
&\leq 2 \|f\ast g^{-\ast}\|_{L^2_s(\p\B)}\|1-g\ast p_n\|_{L^2_s(\p\B)}
\end{align*} 
where the last inequality follows from the Representation Formula \eqref{repfor2} {and Cauchy--Schwarz inequality}. In particular we deduce that
for negative values of $k$, it holds
$$
0=\lim_{n\to+\infty}|\langle f*g^{-*}-f*p_n,q^k \rangle_{ L^2_{s} (\p\B)}|=|\langle f\ast g^{-\ast}, q^k\rangle_{L^2_s(\p\B)}|
$$
and we can conclude that $f\ast g^{-\ast}\in H^2(\p\B)$.

Suppose now that condition $(ii)$ holds, and consider the factorization of $g=\varphi*f$ with $\varphi$ inner and $f$ cyclic.

Thanks to Lemma \ref{modulo}, $\|\varphi^{-*}\|_{L^2_s(\p\B)}=\|\widetilde{\varphi^c}\|_{L^2_s(\p\B)}=\|\varphi\|_{L^2_s(\p\B)}$, so $\varphi^{-*}\in L^2_s(\p\B)$. But $\varphi^{-*}=f*g^{-*}$ which, by $(ii)$, implies that 
$\varphi^{-*}\in H^2(\p\B)$. Thus, both $\varphi$ and $\varphi^{-\ast}$ are in $H^2(\p\B)$. Thanks to \cite[Proposition 5.14]{deFGS} that guarantees that if both a function and its $*$-reciprocal belong to $H^2(\p\B)$, then the function is outer, we have that the inner function $\varphi$ is also an outer function. Therefore, it is a constant $u$ of modulus 1. In particular, $g=u*f$ and hence it is obviously cyclic.

Let us now show that condition $(iii)$ implies condition $(ii)$. Let $f\in H^2(\p\B)$ be such that $f*g^{-*}\in L^2_s(\p\B)$. The fact that $g$ is outer, yields that $g(q)\neq 0$ for any $q\in\B$. Hence $f*g^{-*}$ is a slice regular function in $\B$, bounded in $L^2_s(\p\B)$-norm and thus it belongs to $H^2(\p\B)$. 
At last we show that $(i)$ implies $(iii)$ and this will conclude the proof. Let $f\in H^2(\p\B)$ be such that $|f|=|g|$ $\Sigma$-almost everywhere on $\p\B$. Since $g$ is cyclic, then there exists a sequence $\{p_M\}$ of quaternionic polynomials of the form $p_M(q)=\sum_{n=0}^M q^n \alpha_{M,n}$ such that
\[\lim_{M\to\infty}\|f-g\ast p_M\|_{L_s^2(\p\B)}=0.\]
Then, we can select a subsequence $\{p_{M_j}\}$ such that $f(q)= \lim_{j\to\infty }g\ast p_{M_j}(q)$ for $\Sigma$-almost every $q\in\p\B$. We remark also that $f(q)= \lim_{j\to\infty }g\ast p_{M_j}(q)$ for any $q\in\B$ since $H^2(\p\B)$ is a reproducing kernel Hilbert space, thus
\[|f(q)-g\ast p_{M_j}(q)|\le C\|f-g*p_{M_j}\|_{L^2_s(\p\B)},\]
for some positive constant $C$.
Since $g\in H^2(\p\B)$, then $g$ is $\Sigma$-almost everywhere non-vanishing on $\p\B$  (see \cite[Proposition 4.4]{deFGS}),  thus 
$$
|g(q)|=|f(q)|=|g(q)|\lim_{j\to\infty}|p_{M_j}(T_{g^c}(q))|,
$$
for $\Sigma$-almost every $q\in\p\B$. That is, for every fixed $\varepsilon>0$ there exists $j_0=j_0(\varepsilon)$ such that for every $j>j_0$ it holds $|1-|p_{M_j}(T_{g^c}(q))||<\varepsilon$. Since $g\in H^2(\p\B)$, from \eqref{op-T-phi} and Proposition \ref{zerisimH2}  we deduce that it actually holds $|1-|p_{M_j}(q)||<\varepsilon$ for $\Sigma$-almost every $q\in\p\B$. Finally, exploiting the maximum modulus principle  for slice regular functions  (see \cite{libroGSS}), we get that for any $q\in\B$, it holds
$$
|f(q)|=|g(q)|\lim_{j\to\infty}|p_{M_j}(T_{g^c}(q))|<(1+\varepsilon)|g(q)|.
$$
Since this holds for any $\varepsilon>0$ we finally get $|f(q)|\leq |g(q)|$ for any $q\in\B$ as we wished to show and the proof is concluded.

\end{proof}

The proof of Theorem \ref{thm-factorization} clearly follows from Theorems \ref{thm-fact-cyclic} and \ref{prop-outer}.

\subsection*{Acknowledgments}
{The authors wish to thank  A. Altavilla and C. Stoppato for useful discussions about the topic. The authors wish to thank also the anonymous referees for useful suggestions and for pointing out some references.

The second author wish to thank the Institut Montpelli\'erain Alexandre Grothendieck where part of this project was carried out.}

\bibliography{quaternions-bib}

\providecommand{\bysame}{\leavevmode\hbox to3em{\hrulefill}\thinspace}
\providecommand{\MR}{\relax\ifhmode\unskip\space\fi MR }
\providecommand{\MRhref}[2]{%
  \href{http://www.ams.org/mathscinet-getitem?mr=#1}{#2}
}
\providecommand{\href}[2]{#2}
\begin{thebibliography}{DRGS13}

\bibitem[ABCS15]{milanesi-nevanlinna}
D.~Alpay, V.~Bolotnikov, F.~Colombo, and I.~Sabadini, \emph{Self-mappings of
  the quaternionic unit ball: multiplier properties, the {S}chwarz-{P}ick
  inequality, and the {N}evanlinna-{P}ick interpolation problem}, Indiana Univ.
  Math. J. \textbf{64} (2015), no.~1, 151--180. \MR{3320522}

\bibitem[ACS14]{milanesi-krein}
D.~Alpay, F.~Colombo, and I.~Sabadini, \emph{Krein-{L}anger factorization and
  related topics in the slice hyperholomorphic setting}, J. Geom. Anal.
  \textbf{24} (2014), no.~2, 843--872. \MR{3192300}

\bibitem[ACS16]{libroMilanesi}
\bysame, \emph{Slice hyperholomorphic {S}chur analysis}, Operator Theory:
  Advances and Applications, vol. 256, Birkh\"auser/Springer, Cham, 2016.
  \MR{3585855}

\bibitem[AdF]{adF}
A.~Altavilla and C.~de~Fabritiis, \emph{$*$-exponential of slice-regular
  functions}, preprint.

\bibitem[AL18]{arcozzi-levi}
N.~Arcozzi and M.~Levi, \emph{On a class of shift-invariant subspaces of the
  {D}rury-{A}rveson space}, Concr. Oper. \textbf{5} (2018), 1--8. \MR{3797923}

\bibitem[ARS96]{aleman-richter-sundberg}
A.~Aleman, S.~Richter, and C.~Sundberg, \emph{Beurling's theorem for the
  {B}ergman space}, Acta Math. \textbf{177} (1996), no.~2, 275--310.
  \MR{1440934}

\bibitem[AS15]{metrica}
N.~Arcozzi and G.~Sarfatti, \emph{Invariant metrics for the quaternionic
  {H}ardy space}, J. Geom. Anal. \textbf{25} (2015), no.~3, 2028--2059.
  \MR{3358083}

\bibitem[AS17]{milanesi-beurling}
D.~Alpay and I.~Sabadini, \emph{Beurling-{L}ax type theorems in the complex and
  quaternionic setting}, Linear Algebra Appl. \textbf{530} (2017), 15--46.
  \MR{3672946}

\bibitem[Beu49]{beu}
A.~Beurling, \emph{On two problems concerning linear transformations in hilbert
  space}, Acta Mathematica \textbf{81} (1949), no.~1, 239--255.

\bibitem[dGS18]{deFGS}
C.~{de Fabritiis}, G.~{Gentili}, and G.~{Sarfatti}, \emph{{Quaternionic Hardy
  spaces}}, Ann. Sc. Norm. Super. Pisa Cl. Sci. \textbf{18} (2018), no.~2,
  697--733.

\bibitem[DRGS13]{BlochDGS}
C.~Della~Rocchetta, G.~Gentili, and G.~Sarfatti, \emph{A {B}loch-{L}andau
  theorem for slice regular functions}, Advances in hypercomplex analysis,
  Springer INdAM Ser., vol.~1, Springer, Milan, 2013, pp.~55--74. \MR{3014609}

\bibitem[GMP13]{gmp}
R.~Ghiloni, V.~Moretti, and A.~Perotti, \emph{Continuous slice functional
  calculus in quaternionic {H}ilbert spaces}, Rev. Math. Phys. \textbf{25}
  (2013), no.~4, 1350006, 83. \MR{3062919}

\bibitem[GP11]{ghiloniperotti}
R.~Ghiloni and A.~Perotti, \emph{Slice regular functions on real alternative
  algebras}, Adv. Math. \textbf{226} (2011), no.~2, 1662--1691. \MR{2737796}

\bibitem[GPS17]{gps}
R.~Ghiloni, A.~Perotti, and C.~Stoppato, \emph{The algebra of slice functions},
  Trans. Amer. Math. Soc. \textbf{369} (2017), no.~7, 4725--4762. \MR{3632548}

\bibitem[GS07]{gentili-struppa}
G.~Gentili and D.~C. Struppa, \emph{A new theory of regular functions of a
  quaternionic variable}, Adv. Math. \textbf{216} (2007), no.~1, 279--301.
  \MR{2353257}

\bibitem[GSS13]{libroGSS}
G.~Gentili, C.~Stoppato, and D.~C. Struppa, \emph{Regular functions of a
  quaternionic variable}, Springer Monographs in Mathematics, Springer,
  Heidelberg, 2013. \MR{3013643}

\bibitem[Hel64]{helson}
H.~Helson, \emph{Lectures on invariant subspaces}, Academic Press, New
  York-London, 1964. \MR{0171178}

\bibitem[LR15]{luo-richter}
S.~Luo and S.~Richter, \emph{Hankel operators and invariant subspaces of the
  {D}irichlet space}, J. Lond. Math. Soc. (2) \textbf{91} (2015), no.~2,
  423--438. \MR{3355109}

\bibitem[Nik02]{nikolski}
N.~K. Nikolski, \emph{Operators, functions, and systems: an easy reading.
  {V}ol. 1}, Mathematical Surveys and Monographs, vol.~92, American
  Mathematical Society, Providence, RI, 2002, Hardy, Hankel, and Toeplitz,
  Translated from the French by Andreas Hartmann. \MR{1864396}

\bibitem[Ric88]{richter}
S.~Richter, \emph{Invariant subspaces of the {D}irichlet shift}, J. Reine
  Angew. Math. \textbf{386} (1988), 205--220. \MR{936999}

\bibitem[RS16]{richter-sunkes}
S.~Richter and J.~Sunkes, \emph{Hankel operators, invariant subspaces, and
  cyclic vectors in the {D}rury-{A}rveson space}, Proc. Amer. Math. Soc.
  \textbf{144} (2016), no.~6, 2575--2586. \MR{3477074}

\bibitem[Sar16]{sarfatti-indiana}
G.~Sarfatti, \emph{Quaternionic {H}ankel operators and approximation by slice
  regular functions}, Indiana Univ. Math. J. \textbf{65} (2016), no.~5,
  1735--1757. \MR{3571445}

\bibitem[Wie33]{wiener}
N.~Wiener, \emph{The {F}ourier integral and certain of its applications},
  Cambridge Mathematical Library, Cambridge University Press, Cambridge,
  1988[1933], Reprint of the 1933 edition, With a foreword by Jean-Pierre
  Kahane. \MR{983891}

\end{thebibliography}
\bibliographystyle{amsalpha}
\Addresses
\end{document}